\documentclass[11pt,reqno]{amsart}

\usepackage{amsmath,amssymb,amsthm,mathrsfs,graphicx,xcolor}

\theoremstyle{plain} 
\newtheorem{theorem}{Theorem} 
\newtheorem{lemma}[theorem]{Lemma} 
\newtheorem{corollary}[theorem]{Corollary}

\theoremstyle{definition} 
\newtheorem*{remark}{Remark}

\begin{document}

\title[Stability results]{Stability results for sections \\ of convex bodies}

\author{M.~Stephen and V.~Yaskin}

\address{Matthew Stephen, Department of Mathematical and Statistical Sciences, University of Alberta, Edmonton, Alberta, T6G 2G1, Canada}
\email{mastephe@ualberta.ca}

\address{Vladyslav Yaskin, Department of Mathematical and Statistical Sciences, University of Alberta, Edmonton, Alberta, T6G 2G1, Canada}
\email{yaskin@ualberta.ca}


\thanks{Both authors were partially supported by NSERC}


\begin{abstract}
It is shown by Makai, Martini, and \'Odor that a convex body $K$, all of whose maximal sections pass through the origin, must be origin-symmetric. We prove a stability version of this result. We also discuss a theorem of Koldobsky and Shane about determination of convex bodies by fractional derivatives of the parallel section function and establish the corresponding stability result. 
\end{abstract}

\maketitle

\section{Introduction}

Let $K$ be a {\it convex body} in $\mathbb R^n$, i.e. a compact convex set with non-empty interior. More generally, a {\it body} is a compact subset of $\mathbb{R}^n$ which is equal to the closure of its interior. Throughout the paper, we assume all bodies include the origin as an interior point. Now, we say $K$ is {\it origin-symmetric} if $K=-K$. The {\it parallel section function} of $K$ in the direction $\xi\in S^{n-1}$ is defined by
\begin{equation*}
A_{K,\xi}(t)=\mathrm{vol}_{n-1}(K\cap \{\xi^{\perp}+t\xi\}), \quad t\in \mathbb R.
\end{equation*}
Here, $\xi^{\perp}=\{x\in \mathbb R^n:\, \langle x,\xi\rangle =0\}$ is the hyperplane passing through the origin and orthogonal to the vector $\xi$.

For the study of central sections it is often more natural to consider a larger class of bodies than the class of convex bodies. Recall that if $K$ is a body containing the origin in its interior and star-shaped with respect to the origin, its {\it radial function} is defined by
\begin{align*}
\rho_K(\xi) = \max\{a \ge 0 :a \xi \in K\}, \quad \xi\in S^{n-1}.
\end{align*}
Geometrically, $\rho_K(\xi)$ is the distance from the origin to the point on the boundary in the direction of $\xi$. If $\rho_K$ is continuous, then $K$ is called a {\it star body}. Every convex body (with the origin in its interior) is a star body. The {\it intersection body} of a star body $K$ is the star body $IK$ with radial function
\begin{align*}
\rho_{IK}(\xi) = \mathrm{vol}_{n-1} ( K\cap\xi^\perp ) , \quad \xi\in S^{n-1} .
\end{align*}
Intersection bodies were introduced by Lutwak in \cite{L} and have been actively studied since then. For example, they played a crucial role in the solution of the Busemann-Petty problem (see \cite{Koldobsky} for details).

The {\it cross-section body} of a convex body $K$ is the star body $CK$ with radial function
\begin{align*}
\rho_{CK}(\xi) = \max_{t\in\mathbb{R}} A_{K,\xi}(t), \quad \xi\in S^{n-1} .
\end{align*}
Cross-section bodies were introduced by Martini \cite{Ma}. For properties of these bodies and related questions see \cite{B}, \cite{F}, \cite{GRYZ}, \cite{MMO}, \cite{Me}, \cite{NRZ}.

Brunn's theorem asserts that the origin-symmetry of a convex body $K$ implies
\begin{align*}
A_{K,\xi}(0) =\max_{t\in \mathbb R} A_{K,\xi}(t)
\end{align*}
for all $\xi\in S^{n-1}$. In other words, $CK = IK$. The converse statement was proved by Makai, Martini and \'Odor \cite{MMO}.

\begin{theorem}[Makai, Martini and \'Odor]\label{MMO}
If $K$ is a convex body in $\mathbb R^n$ such that $CK = IK$, then $K$ is origin-symmetric.
\end{theorem}


The goal of the present paper is to provide a stability version of Theorem~\ref{MMO}.  For star bodies $K$ and $L$ in $\mathbb R^n$, the {\it radial metric} is defined as
\begin{align*}
\rho(K,L) = \max_{\xi\in S^{n-1}}|\rho_K(\xi) - \rho_L(\xi)|.
\end{align*}
We prove the following result.

\begin{theorem}\label{MainResult_1}
Let $K$ be a convex body in $\mathbb{R}^n$ contained in a ball of radius $R$, and containing a ball of radius $r$, where both balls are centred at the origin. If there exists $0<\varepsilon  < \min\left\{\left( \frac{\sqrt{3} \, r}{ 6\sqrt{3}\,\pi r + 32\pi } \right)^2 ,\, \frac{r^2}{16} \right\}$ so that
\begin{align*}
\rho(CK, IK) \leq \varepsilon ,
\end{align*}
then
\begin{align*}
\rho(K,-K) \leq C(n,r,R) \, \varepsilon^q \quad \mbox{where} \quad q =
	\begin{cases}
		\frac{1}{2} &\mbox{ if } n=2,\\
		\frac{1}{2(n+1)} &\mbox{ if } n=3,4, \\
		\frac{1}{(n-2)(n+1)} &\mbox{ if } n\geq 5 .
	\end{cases}
\end{align*}
Here, $C(n,r,R)>0$ are constants depending on the dimension, $r$, and $R$.
\end{theorem}

\begin{remark}
In the proof of Theorem \ref{MainResult_1}, we give the explicit dependency of $C(n,r,R)$ on $r$ and $R$.
\end{remark}

The following corollary is a straightforward consequence of the Lipschitz property of the parallel section function (Lemma \ref{Lipschitz})  and Theorem \ref{MainResult_1}. Roughly speaking: if, for every direction $\xi\in S^{n-1}$, the convex body $K$ has a maximal section perpendicular to $\xi$ that is close to the origin, then $K$ is close to being origin-symmetric.

\begin{corollary}\label{Cor 1}
Let $K$ be a convex body in $\mathbb{R}^n$ contained in a ball of radius $R$, and containing a ball of radius $r$, where both balls are centred at the origin. Let $L=L(n)$ be the constant given in Lemma \ref{Lipschitz}. If there exists 
\begin{align*}
0<\varepsilon< \min \left\{ \frac{r}{2}, \, \frac{3 r^3}{L R^{n-1}\left( 6\sqrt{3} \pi r +32\pi\right)^2}, \,
	\frac{r^3}{16 L R^{n-1}} \right\} 
\end{align*}
so that, for each direction $\xi\in S^{n-1}$, $A_{K,\xi}$ attains its maximum at some $t=t(\xi)$ with $|t(\xi)|\leq\varepsilon$, then 
\begin{align*}
\rho(K,-K) \leq \widetilde{C}(n,r,R) \, \varepsilon^q.
\end{align*}
Here, $\widetilde{C}(n,r,R)>0$ are constants depending on the dimension, $r$, and $R$, and $q=q(n)$ is the same as in Theorem \ref{MainResult_1}. 
\end{corollary}

The proof of Theorem \ref{MainResult_1} is given in Section 4 and consists of a sequence of lemmas from Section 3. The main idea is the following. If $K$ is of class $C^\infty$, then we use Brunn's theorem and an integral formula from \cite{BV} to show that $\rho(CK, IK)$ being small implies that $\int_{S^{n-1}} \big| A_{K,\xi}'(0)\big|^2\, d\xi$ is also small. (Recall that $K$ is called {\it $m$-smooth} or {\it $C^m$},  if $\rho_K\in C^m(S^{n-1})$.) If $K$ is not smooth, we approximate it by smooth bodies, for which the above integral is small. Then  we use the Fourier transform techniques from \cite{Rya&Yask2012} and the tools of spherical harmonics similar to those from \cite{Good&Yask2009} to finish the proof.

As we will see below, the same methods can be used to obtain a stability version of a result of Koldobsky and Shane \cite{KS}.
It is well known that the knowledge of  $A_{K,\xi}(0)$ for all $\xi\in S^{n-1}$ is not sufficient for determining the body $K$ uniquely, unless $K$ is origin-symmetric. However, Koldobsky and Shane have shown that if $A_{K,\xi}(0)$ is replaced by  a fractional derivative of non-integer  order of the function  $A_{K,\xi}(t)$ at $t=0$, then this information does determine the body uniquely.

\begin{theorem}[Koldobsky and Shane]\label{KS}
Let $K$ and $L$ be convex bodies in $\mathbb R^n$.
Let $ -1 < p < n-1$ be a non-integer, and $m$ be an integer greater than $p$. If $K$
and $L$ are $m$-smooth and $$A^{(p)}_{K,\xi}(0) = A^{(p)}_{K,\xi}(0),$$ for all $\xi\in S^{n-1}$, then $$K = L.$$
\end{theorem}

The following is our stability result.

\begin{theorem}\label{MainResult_2}
Let $K$ and $L$ be convex bodies in $\mathbb{R}^n$ contained in a ball of radius $R$, and containing a ball of radius $r$, where both balls are centred at the origin. Let $-1<p<n-1$ be a non-integer, and $m$ be an integer greater than $p$. If $K$ and $L$ are $m$-smooth and
\begin{align*}
\sup_{\xi\in S^{n-1}} \Big| A_{K,\,\xi}^{(p)}(0) - A_{L,\,\xi}^{(p)}(0)\Big| \leq \varepsilon
\end{align*}
for some $0 < \varepsilon < 1$, then
\begin{align*}
\rho(K,L) \leq C(n,p,r,R)\,   \varepsilon^q \quad \mbox{where} \quad q =
	\begin{cases}
		\frac{2}{n+1} &\mbox{ if } n\leq 2p+2, \\
		\frac{4}{(n-2p)(n+1)} &\mbox{ if } n>2p+2 .
	\end{cases}
\end{align*}
Here, $C(n,p,r,R)>0$ are constants depending on the dimension, $p$, $r$, and $R$.
\end{theorem}

\begin{remark}
In the proof of Theorem \ref{MainResult_2}, we give the explicit dependency of $C(n,p,r,R)$ on $r$ and $R$. Furthermore, our second result remains true when $p$ is a non-integer greater than $n-1$. However, considering such values for $p$ would make our arguments less clear.
\end{remark}

\section{Preliminaries}

Throughout our paper, the constants
\begin{align*}
\kappa_n := \frac{\pi^\frac{n}{2}}{\Gamma\left( \frac{n}{2} + 1 \right) }
	\qquad \mbox{and} \qquad \omega_n := n\cdot\kappa_n
\end{align*}
give the volume and surface area of the unit Euclidean ball in $\mathbb{R}^n$, where $\Gamma$ denotes the Gamma function. Whenever we integrate over Borel subsets of the sphere $S^{n-1}$, we are using non-normalized {\it spherical measure}; that is, the $(n-1)$-dimensional Hausdorff measure on $\mathbb{R}^n$, scaled so that the measure of $S^{n-1}$ is $\omega_n$.

Let $K$ be a convex body in $\mathbb{R}^n$ containing the origin in its interior. The {\it maximal section function} of $K$ is defined by
\begin{align*}
m_K (\xi) = \max_{t\in\mathbb{R}} \mathrm{vol}_{n-1}(K\cap \{\xi^{\perp}+t\xi\}) = \max_{t\in\mathbb{R}} A_{K,\xi}(t) ,
	\qquad \xi\in S^{n-1} .
\end{align*}
Note that $m_K$ is simply the radial function for the cross-section body $CK$. For each $\xi\in S^{n-1}$, we let $t_K(\xi)\in\mathbb{R}$ be the closest to zero number such that
\begin{align*}
A_{K,\xi}(t_K(\xi)) = m_K(\xi) .
\end{align*}

Towards the proof of our first stability result, we use the formula
\begin{align}\label{BK formula}
\begin{split}
f_{K}(t) :&= \frac{1}{\omega_n} \int_{S^{n-1}} A_{K,\xi}(t) \, d\xi \\
&= \frac{ \Gamma\left( \frac{n}{2} \right)}{\sqrt{\pi} \, \Gamma\left( \frac{n-1}{2} \right) }
	\int_{K\cap \lbrace |x| \geq |t| \rbrace} \frac{1}{|x|} \left( 1 - \frac{t^2}{|x|^2} \right)^\frac{n-3}{2} dx ;
\end{split}
\end{align}
refer to Lemma 1.2 in \cite{BV} or Lemma 1 in \cite{BoK} for the proof.

The {\it Minkowski functional} of $K$ is defined by
\begin{align*}
\|x\|_K=\min \{a\ge 0: x \in aK \}, \qquad x\in \mathbb R^n.
\end{align*}
It easy to see that $\rho_K(\xi) = \|\xi\|_K^{-1}$ for $\xi\in S^{n-1}$. The latter also allows us to consider $\rho_K$ as a homogeneous degree $-1$ function on $\mathbb R^n\setminus\{0\}$. The {\it support function} of $K$ is defined by
\begin{align*}
h_K(x) = \sup_{y\in K} \langle x,y\rangle, \qquad x\in \mathbb R^n.
\end{align*}
The function $h_K$ is the Minkowski functional for the polar body $K^\circ$ associated with $K$. Given another convex body $L$ in $\mathbb{R}^n$, define
\begin{align*}
\delta_2(K,L) = \left( \int_{S^{n-1}} \left| h_K(\xi) - h_L(\xi) \right|^2 \, d\xi \right)^\frac{1}{2}
\end{align*}
and
\begin{align*}
\delta_\infty (K,L) = \sup_{\xi\in S^{n-1}} \left| h_K(\xi) - h_L(\xi) \right| .
\end{align*}
These functions are, respectively, the {\it $L^2$} and {\it Hausdorff metrics} for convex bodies in $\mathbb{bR}^n$. The following theorem, due to Vitale \cite{Vit1985}, relates these metrics; refer to Proposition 2.3.1 in \cite{Groemer} for the proof.

\begin{theorem}\label{Vitale}
Let $K$ and $L$ be convex bodies in $\mathbb{R}^n$, and let $D$ denote the diameter of $K\cup L$. Then
\begin{align*}
\frac{2 \kappa_{n-1} D^{1-n}}{n(n+1)}\, \delta_\infty(K,\,L)^{n+1} \leq \delta_2(K,\, L)^2 \leq \omega_n\, \delta_\infty(K,L)^2 .
\end{align*}
\end{theorem}

Let $\alpha = (\alpha_1,\ldots,\alpha_n)$ be any $n$-tuple of non-negative integers. We will use the notation
\begin{align*}
[\alpha] := \sum_{j=1}^n \alpha_j
\end{align*}
to define the differential operator
\begin{align*}
\frac{\partial^{[\alpha]}}{\partial x^\alpha} := \frac{\partial^{[\alpha]}}{\partial x_1^{\alpha_1}\cdots\partial x_n^{\alpha_n}}.
\end{align*}

We let $\mathcal{S}(\mathbb{R}^n)$ denote the space of Schwartz test functions; that is, functions in $C^\infty(\mathbb{R}^n)$ for which all derivatives decay faster than any rational function. The Fourier transform of $\phi\in\mathcal{S}(\mathbb{R}^n)$ is a test function $\mathscr{F}\phi$ defined by
\begin{align*}
\mathscr{F}\phi (x) = \widehat{\phi}(x) = \int_{\mathbb{R}^n} \phi(y) e^{-i\langle x,y\rangle} \, dy, \qquad x\in\mathbb{R}^n .
\end{align*}
The continuous dual of $\mathcal{S}(\mathbb{R}^n)$ is denoted as $\mathcal{S}'(\mathbb{R}^n)$, and elements of $\mathcal{S}'(\mathbb{R}^n)$ are referred to as distributions. The action of $f\in\mathcal{S}'(\mathbb{R}^n)$ on a test function $\phi$ is denoted as $\langle f,\phi\rangle$. The Fourier transform of $f$ is a distribution $\widehat{f}$ defined by
\begin{align*}
\langle \widehat{f},\phi\rangle = \langle f, \widehat{\phi}\rangle , \qquad \phi\in\mathcal{S}(\mathbb{R}^n);
\end{align*}
$\widehat{f}$ is well-defined as a distribution because $\mathscr{F}:\mathcal{S}(\mathbb{R}^n)\rightarrow\mathcal{S}(\mathbb{R}^n)$ is a continuous and linear bijection.

For any $f\in C(S^{n-1})$ and $p\in\mathbb{C}$, the $-n+p$ homogeneous extension of $f$ is given by
\begin{align*}
f_p(x) = \vert x\vert^{-n+p}\, f\left(\frac{x}{\vert x\vert}\right), \qquad x\in\mathbb{R}^n\backslash \lbrace 0\rbrace.
\end{align*}
When $\mathcal{R}p>0$, $f_p$ is locally integrable on $\mathbb{R}^n$ with at most polynomial growth at infinity. In this case, $f_p$ is a distribution on $\mathcal{S}(\mathbb{R}^n)$ acting by integration, and we may consider its Fourier transform. Goodey, Yaskin, and Yaskina show in \cite{Good&Yask2009} that, for $f\in C^\infty (S^{n-1})$, the additional restriction $\mathcal{R}p<n$ ensures the action of $\widehat{f_p}$ is also by integration, with $\widehat{f_p}\in C^\infty ( \mathbb{R}^n\backslash\lbrace 0\rbrace)$.

We make extensive use of the mapping $I_p: C^\infty (S^{n-1})\rightarrow C^\infty ( S^{n-1})$ defined in \cite{Good&Yask2009}, which sends a function $f$ to the restriction of $\widehat{f_p}$ to $S^{n-1}$. For $0<\mathcal{R}p<n$ and $m\in\mathbb{Z}^{\geq 0}$, Goodey, Yaskin and Yaskina show $I_p$ has an eigenvalue $\lambda_m(n,p)$ whose eigenspace includes all spherical harmonics of degree $m$ and dimension $n$. These eigenvalues are given explicitly in the following lemma; refer to \cite{Good&Yask2009} for the proof.

\begin{lemma}\label{eigen}
If $0<\mathcal{R}p<n$, then the eigenvalues $\lambda_m(n,p)$ are given by
\begin{align*}
\lambda_m(n,p) = \frac{2^p\pi^\frac{n}{2}(-1)^\frac{m}{2}\Gamma\left(\frac{m+p}{2}\right)}{\Gamma\left( \frac{m+n-p}{2}\right)}
	\mbox{ \ if m is even},
\end{align*}
and
\begin{align*}
\lambda_m(n,p) = i\,\frac{2^p\pi^\frac{n}{2}(-1)^\frac{m-1}{2}\Gamma\left(\frac{m+p}{2}\right)}{\Gamma\left( \frac{m+n-p}{2}\right)}
	\mbox{ \ if m is odd}.
\end{align*}
\end{lemma}

The {\it spherical gradient} of $f\in C(S^{n-1})$ is the restriction of $\nabla f\left(^x/_{|x|}\right)$ to $S^{n-1}$. It is denoted by $\nabla_o f$.

An extensive discussion on spherical harmonics is given in \cite{Groemer}. A {\it spherical harmonic} $Q$ of dimension $n$ is a harmonic and homogeneous polynomial in $n$ variables whose domain is restricted to $S^{n-1}$. We say $Q$ is of degree $m$ if the corresponding polynomial has degree $m$. The collection $\mathcal{H}_m^n$ of all spherical harmonics with dimension $n$ and degree $m$ is a finite dimensional Hilbert space with respect to the inner product for $L^2(S^{n-1})$. If, for each $m\in\mathbb{Z}^{\geq 0}$, $\mathcal{B}_m$ is an orthonormal basis for $\mathcal{H}_m^n$, then the union of all $\mathcal{B}_m$ is an orthonormal basis for $L^2(S^{n-1})$. Given $f\in L^2(S^{n-1})$, and defining
\begin{align*}
\sum_{Q\in\mathcal{B}_m} \langle f,Q\rangle\, Q  =: Q_m\in\mathcal{H}_m^n,
\end{align*}
we call $\sum_{m=0}^\infty Q_m$ the {\it condensed harmonic expansion} for $f$. The condensed harmonic expansion does not depend on the particular orthonormal bases chosen for each $\mathcal{H}_m^n$.

Let $m\in\mathbb{N}\cup\lbrace 0\rbrace$, and let $h:\mathbb{R}\rightarrow\mathbb{C}$ be an integrable function which is $m$-smooth in a neighbourhood of the origin. For $p\in\mathbb{C}\backslash\mathbb{Z}$ such that $-1<\mathcal{R} p < m$, we define the {\it fractional derivative} of the order $p$ of $h$ at zero as
\begin{align*}
h^{(p)}(0) = &\frac{1}{\Gamma(-p)} \int_0^1 t^{-1-p} \left( h(-t)
	- \sum_{k=0}^{m-1} \frac{ (-1)^k h^{(k)}(0) }{k!} t^k \right) \, dt \\
&+ \frac{1}{\Gamma(-p)} \int_1^\infty t^{-1-p} h(-t)\, dt
	+ \frac{1}{\Gamma(-p)} \sum_{k=0}^{m-1} \frac{ (-1)^k h^{(k)}(0) }{k!(k-p)} .
\end{align*}
Given the simple poles of the Gamma function, the fractional derivatives of $h$ at zero may be analytically extended to the integer values $0,\ldots, m-1$, and they will agree with the classical derivatives. 

Let $K$ be an infinitely smooth convex body. By Lemma 2.4 in \cite{Koldobsky}, $A_{K,\xi}$ is infinitely smooth in a neighbourhood of $t=0$ which is uniform with respect to $\xi\in S^{n-1}$. With the exception of a sign difference, the equality
\begin{align}\label{PS Func}
A_{K,\,\xi}^{(p)}(0) = & \frac{\cos\left(\frac{p\pi}{2}\right)}{2\pi(n-1-p)}\, \Big(\Vert x\Vert_K^{-n+1+p}
	+ \Vert -x\Vert_K^{-n+1+p} \Big)^\wedge (\xi) \\
&+ i\, \frac{\sin\left(\frac{p\pi}{2}\right)}{2\pi(n-1-p)}\, \Big(\Vert x\Vert_K^{-n+1+p}
	- \Vert -x\Vert_K^{-n+1+p} \Big)^\wedge (\xi), \nonumber
\end{align}
was proven by Ryabogin and Yaskin in \cite{Rya&Yask2012} for all $\xi\in S^{n-1}$ and $p\in\mathbb{C}$ such that $-1 < Re(p) < n-1$. The sign difference results from their use of $h(x)$ rather than $h(-x)$ in the definition of fractional derivatives.

\section{Auxiliary Results}

We first prove some auxiliary lemmas.

\begin{lemma}\label{approx}
Let $m$ be a non-negative integer. Let $K$ be an $m$-smooth convex body in $\mathbb{R}^n$ contained in a ball of radius $R$, and containing a ball of radius $r$, where both balls are centred at the origin. There exists a family $\lbrace K_\delta\rbrace_{0<\delta<1}$ of infinitely smooth convex bodies in $\mathbb{R}^n$ which approximate $K$ in the radial metric as $\delta$ approaches zero, with
\begin{align*}
B_{(1+\delta)^{-1}r}(0) \subset K_\delta \subset B_{(1-\delta)^{-1}R}(0) .
\end{align*}
Furthermore,
\begin{align*}
\lim_{\delta\rightarrow 0}\sup_{\xi\in S^{n-1}} \sup_{|t|\leq \frac{r}{4}} \big| A_{K,\xi}(t) - A_{K_\delta,\xi}(t)\big| = 0 ,
\end{align*}
and
\begin{align*}
\lim_{\delta\rightarrow 0}\sup_{\xi\in S^{n-1}} \left| A_{K_\delta,\xi}^{(p)}(0) - A_{K,\xi}^{(p)}(0)\right| = 0
\end{align*}
for every $p\in\mathbb{R}$, $-1< p\leq m$.
\end{lemma}

\begin{proof}
For each $0<\delta<1$, let $\phi_\delta: [0,\infty)\rightarrow [0,\infty)$ be a $C^\infty$ function with support contained in $[\delta/2,\delta]$, and
\begin{align*}
\int_{\mathbb{R}^n} \phi_\delta\big(|z|\big)\, dz = 1.
\end{align*}
It follows from Theorem 3.3.1 in \cite{Schneider} that there is a family $\lbrace K_\delta\rbrace_{0<\delta<1}$ of $C^\infty$ convex bodies in $\mathbb{R}^n$ such that
\begin{align*}
\Vert x\Vert_{K_\delta} = \int_{\mathbb{R}^n} \big\Vert x+|x|z\big\Vert_K \phi_\delta\big(|z|\big)\, dz,
\end{align*}
and
\begin{align*}
\lim_{\delta\rightarrow 0} \sup_{\xi\in S^{n-1}} \big| \Vert\xi\Vert_{K_\delta} - \Vert\xi\Vert_K \big| = 0 .
\end{align*}
For each $\xi\in S^{n-1}$ and $z\in\mathbb{R}^n$ with $|z|\leq\delta$, we have
\begin{align*}
\big\Vert \xi+|\xi|z\big\Vert_K = \Vert \xi + z\Vert_K = \Vert \lambda\eta\Vert_K = \lambda \Vert\eta\Vert_K
\end{align*}
for some $\eta\in S^{n-1}$ and $0< 1-\delta \leq \lambda\leq 1+\delta$. It then follows from the support of $\phi_\delta$ and the inequality $R^{-1}\leq \Vert\eta\Vert_K\leq r^{-1}$ that
\begin{align*}
\Vert\xi\Vert_{K_\delta} = \int_{\mathbb{R}^n} \Vert \xi + z\Vert_K \phi_\delta\big( |z|\big)\, dz \leq (1+\delta) r^{-1}
\end{align*}
and
\begin{align*}
\Vert\xi\Vert_{K_\delta} = \int_{\mathbb{R}^n} \Vert \xi + z\Vert_K \phi_\delta\big( |z|\big)\, dz \geq (1-\delta) R^{-1},
\end{align*}
which gives
\begin{align*}
B_{(1+\delta)^{-1}r}(0) \subset K_\delta \subset B_{(1-\delta)^{-1}R}(0).
\end{align*}
This containment, with the limit of the difference of Minkowski functionals above, implies
\begin{align}\label{radial}
\lim_{\delta\rightarrow 0} \sup_{\xi\in S^{n-1}} \big| \rho_{K_\delta}(\xi) - \rho_K(\xi)\big| = 0 .
\end{align}
Therefore, $\lbrace K_\delta\rbrace_{0<\delta<1}$ approximate $K$ with respect to the radial metric.

Furthermore, the radial functions $\lbrace \rho_{K_\delta} \rbrace_{0<\delta<1}$ approximate $\rho_K$ in $C^m(S^{n-1})$. Let $\alpha = (\alpha_1,\ldots,\alpha_n)$ be any $n$-tuple of non-negative integers such that $1\leq [\alpha]\leq m$, and consider the function
\begin{align*}
f(y,z) := \frac{\partial^{[\alpha]}}{\partial x^\alpha} \big\Vert x+|x|z\big\Vert_K\Big|_{x=y} .
\end{align*}
Observe that $f$ is uniformly continuous on
\begin{align*}
\left\lbrace y\in\mathbb{R}^n, 2^{-1} \leq |y|\leq 2 \right\rbrace
	\times \left\lbrace z\in\mathbb{R}^n, |z|\leq 2^{-1} \right\rbrace
\end{align*}
since $K$ is $m$-smooth. Therefore, we have
\begin{align*}
\frac{\partial^{[\alpha]}}{\partial x^\alpha} \big( \Vert x\Vert_{K_\delta} - \Vert x\Vert_K\big) \Big|_{x=\xi}
	= \int_{\mathbb{R}^n} \frac{\partial^{[\alpha]}}{\partial x^\alpha} \Big( \big\Vert x+|x|z\big\Vert_K
	- \Vert x\Vert_K \Big)\Big|_{x=\xi} \phi_\delta\big(|z|\big)\, dz
\end{align*}
for all $\xi\in S^{n-1}$ and $\delta<1/2$, which implies
\begin{align*}
\sup_{\xi\in S^{n-1}} \left| \frac{\partial^{[\alpha]}}{\partial x^\alpha}
	\big( \Vert x\Vert_{K_\delta} - \Vert x\Vert_K\big) \Big|_{x=\xi} \right|
	\leq \sup_{\xi\in S^{n-1}} \sup_{|z|<\delta} \big| f(\xi,z) - f(\xi,0) \big| .
\end{align*}
Noting that $\big| (\xi,z) - (\xi,0)\big| = |z|<\delta$, the uniform continuity of $f$ then implies
\begin{align}\label{Mink deriv}
\lim_{\delta\rightarrow 0} \sup_{\xi\in S^{n-1}} \left| \frac{\partial^{[\alpha]}}{\partial x^\alpha}
	\big( \Vert x\Vert_{K_\delta} - \Vert x\Vert_K\big) \Big|_{x=\xi} \right| = 0 .
\end{align}
It follows from the relation $\rho_K(x) = \Vert x\Vert_K^{-1}$ that $\frac{\partial^{[\alpha]}}{\partial x^\alpha} \rho_K \big|_{x=\xi}$ may be expressed as a finite linear combination of terms of the form
\begin{align*}
\rho_K^{d+1}(\xi) \prod_{j=0}^d \frac{\partial^{[\beta_j]}}{\partial x^{\beta_j}} \Vert x\Vert_K \Big|_{x=\xi},
\end{align*}
where $d\in\mathbb{Z}^{\geq 0}$, and each $\beta_j$ is an $n$-tuple of non-negative integers such that $[\beta_j]\geq 1$ and
$[\alpha] = \sum_{j=0}^d [\beta_j]$. Of course, $\frac{\partial^{[\alpha]}}{\partial x^\alpha} \rho_{K_\delta}\big|_{x=\xi}$ may be expressed similarly. Equations (\ref{radial}) and (\ref{Mink deriv}) then imply
\begin{align}\label{radial deriv}
\lim_{\delta\rightarrow 0} \sup_{\xi\in S^{n-1}} \left| \frac{\partial^{[\alpha]}}{\partial x^\alpha}
	\big( \rho_{K_\delta} - \rho_K\big) \Big|_{x=\xi} \right| = 0,
\end{align}
once we note that $\rho_K$ and the partial derivatives of $\Vert x\Vert_K$, up to order $m$, are bounded on $S^{n-1}$.

Our next step is to uniformly approximate the parallel section function $A_{K,\xi}$.  Fix $\xi\in S^{n-1}$, and define the hyperplane
\begin{align*}
H_t = \xi^\perp + t\xi
\end{align*}
for any $t\in\mathbb{R}$ such that $|t|<r$. Let $S^{n-2}$ denote the Euclidean sphere in $H_t$ centred at $t\xi$, and let $\rho_{K\cap H_t}$ denote the radial function for $K\cap H_t$ with respect to $t\xi$ on $S^{n-2}$. Then, for $|t|<r$,
\begin{align}\label{int exp}
A_{K,\xi}(t) = \frac{1}{n-1} \int_{S^{n-2}} \rho_{K\cap H_t}^{n-1}(\theta)\, d\theta .
\end{align}
For $|t|<r/2$ and $0<\delta<1$, $A_{K_\delta,\xi}(t)$ may be expressed similarly. Fixing $\theta\in S^{n-2}$, and with angles $\alpha$ and $\beta$ as in Figure \ref{KcapHt}, we have
\begin{align*}
\big| \rho_{K\cap H_t}(\theta) - \rho_{K_\delta\cap H_t}(\theta) \big|
	\leq \frac{\sin\beta}{\sin\alpha} \big| \rho_K(\eta_1) - \rho_{K_\delta}(\eta_1) \big| .
\end{align*}
By restricting to $|t|\leq r/4$, $\alpha$ may be bounded away from zero and $\pi$. Indeed, if $\alpha < \pi/2$, then
\begin{align*}
\tan \alpha \geq \frac{r/2 - |t|}{R} \geq \frac{r}{4R} ,
\end{align*}
and if $\alpha > \pi/2$, then
\begin{align*}
\tan ( \pi-\alpha ) \geq \frac{r/2 + |t|}{R} \geq \frac{r}{2R} .
\end{align*}
Therefore
\begin{align*}
0 < \arctan\left(\frac{r}{4R}\right) \leq \alpha \leq \pi - \arctan\left(\frac{r}{4R}\right) < \pi.
\end{align*}
We now have
\begin{align}\label{rhoKcapHt}
\big| \rho_{K\cap H_t}(\theta) - \rho_{K_\delta\cap H_t}(\theta) \big|
	\leq \frac{1}{\sin\left(\arctan\left(\frac{r}{4R}\right)\right)}
	\sup_{\eta\in S^{n-1}} \big| \rho_K(\eta) - \rho_{K_\delta}(\eta) \big| ,
\end{align}
where the upper bound is independent of $\xi\in S^{n-1}$, $t$ with $|t|\leq r/4$, and $\theta\in S^{n-2}$. This inequality, the integral expression (\ref{int exp}), and equation (\ref{radial}) imply
\begin{align*}
\lim_{\delta\rightarrow 0}\sup_{\xi\in S^{n-1}} \sup_{|t|\leq \frac{r}{4}} \big| A_{K,\xi}(t) - A_{K_\delta,\xi}(t)\big| = 0 .
\end{align*}

\begin{figure}
\center \includegraphics[scale=0.75]{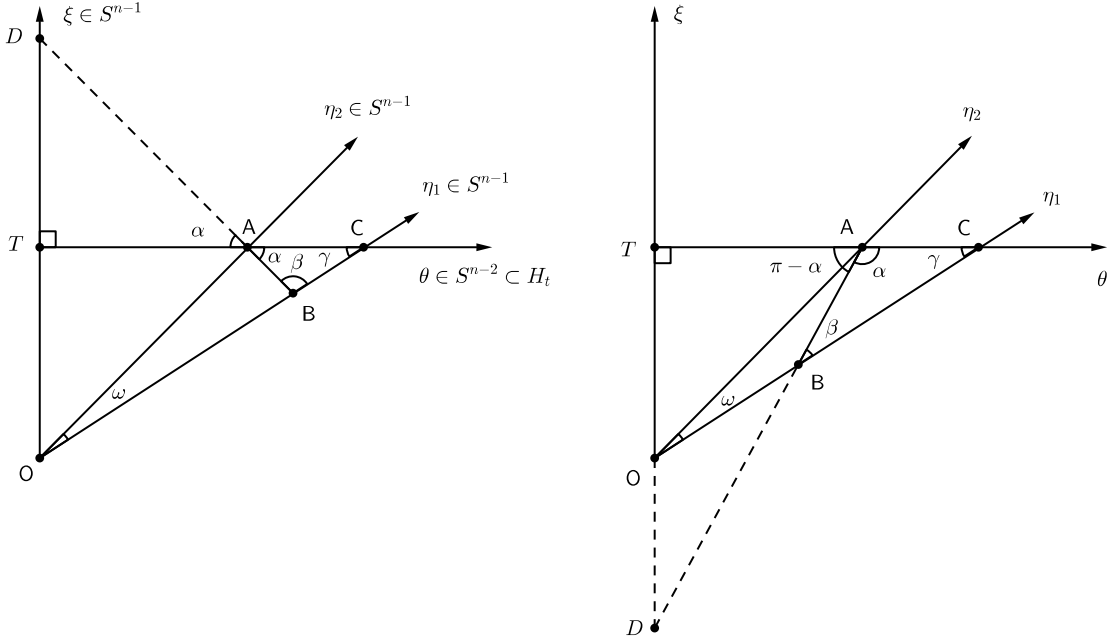}
\caption{ The diagrams represent two extremes: when the angle $\alpha$ is small $(\alpha<\pi/2)$, and when it is large $(\alpha>\pi/2)$. The point $O$ represents the origin in $\mathbb{R}^n$, and $\left|\overline{OT}\right| = t$ where $0\leq t\leq r/4$. The points $A$ and $C$ are the boundary points for $K$ and $K_\delta$ in the direction $\theta$, with two obvious possibilities: either $\left|\overline{TA}\right| = \rho_{K\cap H_t}(\theta)$ and $\left|\overline{TC}\right| = \rho_{K_\delta\cap H_t}(\theta)$, or the opposite. The point $B$ is a boundary point for the same convex body as $A$, but in the direction $\eta_1$. The point $D$ lies outside of the convex body for which $A$ and $B$ are boundary points.}
\label{KcapHt}
\end{figure}

Lemma (2.4) in \cite{Koldobsky} establishes the existence of a small neighbourhood of $t=0$, independent of $\xi\in S^{n-1}$, on which $A_{K,\xi}$ is $m$-smooth. The following is an elaboration of Koldobsky's proof, so that we may uniformly approximate the derivatives of $A_{K,\xi}$. Again fix $\xi\in S^{n-1}$, and fix $\theta\in S^{n-2}\subset H_t$. Let $\rho_{K,\theta}$ denote the $m$-smooth restriction of $\rho_K$ to the two dimensional plane spanned by $\xi$ and $\theta$, and consider $\rho_{K,\theta}$ as a function on $[0,2\pi]$, where the angle is measured from the positive $\theta$-axis. A right triangle then gives the equation
\begin{align*}
\rho_{K\cap H_t}^2(\theta) + t^2 = \rho_{K,\theta}^2\left( \arctan\left(\frac{t}{\rho_{K\cap H_t}(\theta)}\right)\right) ,
\end{align*}
which we can use to implicitly differentiate $y(t):=\rho_{K\cap H_t}(\theta)$ as a function of $t$. Indeed,
\begin{align*}
F(t,y) := y^2 + t^2 - \rho_{K,\theta}^2\left( \arctan\left(\frac{t}{y}\right)\right)
\end{align*}
is differentiable away from $y=0$, with
\begin{align*}
F_y(t,y) = 2y + \frac{2t}{y^2+t^2}\rho_{K,\theta}\left(\arctan\left(\frac{t}{y}\right)\right)
	\rho_{K,\theta}'\left(\arctan\left(\frac{t}{y}\right)\right) .
\end{align*}
The containment $B_r(0)\subset K\subset B_R(0)$ implies $\rho_{K,\theta}$ is bounded above on $S^{n-1}$ by $R$, and
\begin{align*}
\rho_{K\cap H_t}(\theta) \geq \frac{ \sqrt{15} \, r }{4}
\end{align*}
for $|t|\leq r/4$. If
\begin{align*}
M = 1 + \sup_{\xi\in S^{n-1}} \big| \nabla_o \rho_K(\xi) \big| < \infty ,
\end{align*}
and $\lambda\in\mathbb{R}$ is a constant such that
\begin{align*}
0 < \lambda < \min\left\lbrace \frac{15\sqrt{15} \, r^3}{128 RM}, \frac{r}{4} \right\rbrace ,
\end{align*}
then
\begin{align*}
\Big| F_y\big( t, \rho_{K\cap H_t}(\theta) \big) \Big| > \frac{ \sqrt{15} \, r }{4}
\end{align*}
for $|t|\leq\lambda$. Therefore, by the Implicit Function Theorem, $y(t) = \rho_{K\cap H_t}(\theta)$ is differentiable on $(-\lambda,\lambda)$, with
\begin{align*}
y'(t) = \frac{ \rho_{K,\theta}\left(\arctan\left(\frac{t}{y}\right)\right)
	\rho'_{K,\theta}\left(\arctan\left(\frac{t}{y}\right)\right) \big( y^2+t^2\big)^{-1} y - t }
	{ y + t \rho_{K,\theta}\left(\arctan\left(\frac{t}{y}\right)\right)
	\rho'_{K,\theta}\left(\arctan\left(\frac{t}{y}\right)\right) \big( y^2+t^2\big)^{-1} } .
\end{align*}
Recursion shows that $\rho_{K\cap H_t}(\theta)$ is $m$-smooth on $(-\lambda,\lambda)$, independent of $\xi\in S^{n-1}$ and $\theta\in S^{n-2}$. It follows from the integral expression (\ref{int exp}) that $A_{K,\xi}$ is $m$-smooth on $(-\lambda,\lambda)$ for every $\xi\in S^{n-1}$. This argument also shows that $A_{K_\delta,\xi}$ is $m$-smooth on the same interval, for $\delta>0$ small enough. Using the resulting expressions for the derivatives of $A_{K,\xi}$ and $A_{K_\delta,\xi}$, and applying equations (\ref{radial}), (\ref{radial deriv}), and the inequality (\ref{rhoKcapHt}), we have
\begin{align*}
\lim_{\delta\rightarrow 0} \sup_{\xi\in S^{n-1}} \sup_{|t|\leq\lambda}
	\left| A_{K,\xi}^{(k)}(t) - A_{K_\delta,\xi}^{(k)}(t)\right| = 0
\end{align*}
for $k = 1,\ldots,m$.

Finally, for any $p\in\mathbb{R}$ such that $-1<p< m$ and $p\neq 0,1,\ldots, m-1$, we will uniformly approximate $A_{K,\xi}^{(p)}(0)$. With $\lambda>0$ as chosen above, we have
\begin{align*}
A_{K,\xi}^{(p)}(0) = &\frac{1}{\Gamma(-p)} \int_0^\lambda t^{-1-p} \left( A_{K,\xi}(-t)
	- \sum_{k=0}^{m-1} \frac{(-1)^k A_{K,\xi}^{(k)}(0)}{k!} t^k \right)\, dt \\
&+ \frac{1}{\Gamma(-p)} \int_\lambda^\infty t^{-1-p} A_{K,\xi}(-t) \, dt
	+ \frac{1}{\Gamma(-p)} \sum_{k=0}^{m-1} \frac{(-1)^k \lambda^{k-p} A_{K,\xi}^{(k)}(0)}{k!(k-p)} .
\end{align*}
The first integral in this equation can be rewritten as
\begin{align*}
\int_0^\lambda t^{-1-p} \int_0^t \frac{A_{K,\xi}^{(m)}(-z)}{(m-1)!} (t-z)^{m-1}\, dz\, dt ,
\end{align*}
using the integral form of the remainder in Taylor's Theorem. We also have
\begin{align*}
&\int_\lambda^\infty t^{-1-p} A_{K,\xi}(-t) \, dt \\
&= \int_{K\cap \lbrace \langle x,-\xi\rangle \geq\lambda\rbrace} \langle x,-\xi\rangle^{-1-p}\, dx \\
&= \int_{B_K(\xi)} \langle\eta,-\xi\rangle^{-1-p}
	\int_{\lambda\langle\eta,-\xi\rangle^{-1}}^{\rho_K(\eta)} r^{n-2-p}\, dr\,d\eta \\
&= \frac{1}{n-1-p} \int_{B_K(\xi)} \Big( \langle\eta,-\xi\rangle^{-1-p} \rho_K^{n-1-p}(\eta)
	- \lambda^{n-1-p} \langle\eta,-\xi\rangle^{-n} \Big) \, d\eta ,
\end{align*}
where
\begin{align*}
B_K(\xi) = \Big\lbrace \eta\in S^{n-1} \Big|
	\langle \eta,\xi\rangle<0 \mbox{ and } \rho_K(\eta) \geq \lambda\langle\eta,-\xi\rangle^{-1}\Big\rbrace .
\end{align*}
Therefore, with the set $B_{K_\delta}(\xi)$ defined similarly, we have
\begin{align}
&\left| A_{K,\xi}^{(p)}(0) - A_{K_\delta,\xi}^{(p)}(0) \right| \cdot \big| \Gamma(-p) \big| \nonumber \\
&\leq \frac{1}{(m-1)!} \left( \sup_{|z|\leq\lambda} \left| A_{K,\xi}^{(m)}(z) - A_{K_\delta,\xi}^{(m)}(z) \right| \right)
	\int_0^\lambda \int_0^t t^{-1-p} (t-z)^{m-1} \, dz\, dt \label{first} \\
&\quad + \left( \sup_{\eta\in S^{n-1}} \left| \rho_K^{n-1-p}(\eta) - \rho_{K_\delta}^{n-1-p}(\eta) \right| \right)
	\int_{B_K(\xi)\cap B_{K_\delta}(\xi)} \frac{ \langle\eta,-\xi\rangle^{-1-p} }{ |n-1-p| } \, d\eta \label{second} \\
&\quad + \int_{B_K(\xi) \setminus B_{K_\delta}(\xi)} \left| \frac{ \langle\eta,-\xi\rangle^{-1-p} \rho_K^{n-1-p}(\eta)
	- \lambda^{n-1-p} \langle\eta,-\xi\rangle^{-n} }{n-1-p} \right| \, d\eta \label{third} \\
&\quad + \int_{B_{K_\delta}(\xi) \setminus B_K(\xi)} \left| \frac{ \langle\eta,-\xi\rangle^{-1-p} \rho_{K_\delta}^{n-1-p}(\eta)
	- \lambda^{n-1-p} \langle\eta,-\xi\rangle^{-n} }{n-1-p} \right| \, d\eta \label{fourth} \\
&\quad + \sum_{k=0}^{m-1} \frac{\lambda^{k-p}}{k! |k-p|} \left| A_{K,\xi}^{(k)}(0) - A_{K_\delta,\xi}^{(k)}(0) \right| \nonumber ,
\end{align}
for $\delta>0$ small enough. The integrals in expressions (\ref{first}) and (\ref{second}) are finite, with
\begin{align*}
\int_0^\lambda \int_0^t t^{-1-p} (t-z)^{m-1} \, dz\, dt = \frac{\lambda^{m-p}}{m(m-p)} ,
\end{align*}
since $p$ is a non-integer less than $m$, and
\begin{align*}
\int_{B_K(\xi)\cap B_{K_\delta}(\xi)} \langle\eta,-\xi\rangle^{-1-p} \, d\eta
	\leq \left( \frac{R}{\lambda} \right)^{1+p} \omega_n .
\end{align*}
Furthermore, the integrands in expression (\ref{third}) and (\ref{fourth}) are bounded above by
\begin{align*}
\left(\frac{2R}{\lambda}\right)^{1+p} (2R)^{n-1-p} + \lambda^{n-1-p}\left(\frac{2R}{\lambda}\right)^n \qquad \mbox{ if } p<n-1 ,
\end{align*}
and
\begin{align*}
\left(\frac{2R}{\lambda}\right)^{1+p} \left(\frac{r}{2}\right)^{n-1-p} + \lambda^{n-1-p}\left(\frac{2R}{\lambda}\right)^n
	\qquad \mbox{ if } p>n-1 ,
\end{align*}
noting that $B_{r/2}(0)\subset K_\delta\subset B_{2R}(0)$ for $\delta<1/2$.

It is now sufficient to prove
\begin{align*}
\lim_{\delta\rightarrow 0} \sup_{\xi\in S^{n-1}} \int_{S^{n-1}} \chi_{ B(\xi,\delta) } \, d\eta = 0 ,
\end{align*}
where
\begin{align*}
&B(\xi,\delta) = B_K(\xi) \Delta B_{K_\delta}(\xi) \\
&= \left\lbrace \eta\in S^{n-1} \, \bigg| \, \rho_K(\eta) \geq \frac{\lambda}{\langle\eta, -\xi\rangle} > \rho_{K_\delta}(\eta)
	\mbox{ or } \rho_{K_\delta}(\eta) \geq \frac{\lambda}{\langle\eta, -\xi\rangle} > \rho_K(\eta) \right\rbrace .
\end{align*}
We will prove the equivalent statement
\begin{align*}
\lim_{\delta\rightarrow 0} \sup_{\xi\in S^{n-1}} \int_{S^{n-1}} \chi_{ B(-\xi,\delta) } \, d\eta = 0 ,
\end{align*}
where the sign of $\xi$ has changed, so that we may use Figure \ref{KcapHt}.

Towards this end, fix any $\theta\in S^{n-2}$, and consider Figure \ref{KcapHt} specifically when $t = \lambda$. In this case,
\begin{align*}
\left|\overline{OA}\right| = \rho_K(\eta_2) = \lambda\langle\eta_2,\xi\rangle^{-1} \mbox{ and }
	\left|\overline{OC}\right| = \rho_{K_\delta}(\eta_1) = \lambda\langle\eta_1,\xi\rangle^{-1}
\end{align*}
or
\begin{align*}
\left|\overline{OC}\right| = \rho_K(\eta_2) = \lambda\langle\eta_2,\xi\rangle^{-1} \mbox{ and }
	\left|\overline{OA}\right| = \rho_{K_\delta}(\eta_1) = \lambda\langle\eta_1,\xi\rangle^{-1} .
\end{align*}
Any $\eta\in B(-\xi,\delta)$ lying in the right half-plane spanned by $\xi$ and $\theta$ will lie between $\eta_1$ and $\eta_2$. Furthermore, the angle $\omega$ converges to zero as $\delta$ approaches zero, uniformly with respect to $\xi\in S^{n-1}$ and $\theta\in S^{n-2}$. Indeed, we have
\begin{align*}
0\leq \sin\omega \leq \frac{2 \, \sin\beta \, \sin\gamma}{r \, \sin\alpha} \big| \rho_K(\eta_1) - \rho_{K_\delta}(\eta_1) \big| ,
\end{align*}
using the fact that both $K$ and $K_\delta$ contain a ball of radius $r/2$, and with $\sin\alpha$ uniformly bounded away from zero as before. It follows that the spherical measure of $B(-\xi,\delta)$ converges to zero as $\delta$ approaches zero, uniformly with respect to $\xi\in S^{n-1}$.
\end{proof}


\begin{lemma}\label{Lipschitz}
Let $K\subset\mathbb{R}^n$ be a convex body contained in a ball of radius $R$, and containing a ball of radius $r$, where both balls are centred at the origin. If
\begin{align*}
L(n) = 8 (n-1) \pi^\frac{n-1}{2} \left[ \Gamma\left( \frac{n+1}{2} \right) \right]^{-1} ,
\end{align*}
then
\begin{align*}
\left| A_{K,\xi}(t) - A_{K,\xi}(s) \right| \leq L(n) \, R^{n-1} \, r^{-1} \, |t-s|
\end{align*}
for all $s,t\in [-r/2,r/2]$ and $\xi\in S^{n-1}$.
\end{lemma}

\begin{proof}
For $\xi\in S^{n-1}$, Brunn's Theorem implies $f := A_{K,\xi}^\frac{1}{n-1}$ is concave on its support, which includes the interval $[-r,r]$. Let
\begin{align*}
L_0 = \max \left\lbrace \left| \frac{ f\left(\frac{-3r}{4}\right) - f(-r) }{ \frac{-3r}{4} - (-r) } \right| ,
	\left| \frac{ f(r) - f\left(\frac{3r}{4}\right) }{ r - \frac{3r}{4} } \right| \right\rbrace ,
\end{align*}
and suppose $s,t\in [-r/2, r/2]$ are such that $s<t$. If
\begin{align*}
\frac{f(t)-f(s)}{t-s} > 0 ,
\end{align*}
then
\begin{align*}
\frac{ f\left(\frac{-3r}{4}\right) - f(-r) }{ \frac{-3r}{4} - (-r) }
	\geq \frac{ f(s) - f\left( \frac{-3r}{4} \right) }{s- \left( \frac{-3r}{4} \right) }
	\geq \frac{f(t)-f(s)}{t-s} > 0 \, ;
\end{align*}
otherwise, we will obtain a contradiction of the concavity of $f$. Similarly, if
\begin{align*}
\frac{f(t)-f(s)}{t-s} < 0 ,
\end{align*}
then
\begin{align*}
\frac{ f(r) - f\left(\frac{3r}{4}\right) }{ r - \frac{3r}{4} }
	\leq \frac{ f\left( \frac{3r}{4} \right) - f(t) }{ \frac{3r}{4} - t }
	\leq \frac{f(t)-f(s)}{t-s} < 0 .
\end{align*}
Therefore,
\begin{align*}
\left| A_{K,\xi}^\frac{1}{n-1}(t) - A_{K,\xi}^\frac{1}{n-1}(s) \right| \leq L_0 \, |t-s|
\end{align*}
for all $s,t\in [-r/2,r/2]$. Now, we have
\begin{align*}
\left| A_{K,\xi}(t) - A_{K,\xi}(s) \right| &\leq (n-1) \left( \max_{t_0\in\mathbb{R}} A_{K,\xi}(t_0) \right)^\frac{n-2}{n-1}
	\left| A_{K,\xi}^\frac{1}{n-1}(t) - A_{K,\xi}^\frac{1}{n-1}(s) \right|
\end{align*}
by the Mean Value Theorem, and
\begin{align*}
L_0 &\leq \frac{4}{r} \cdot 2 \left( \max_{t_0\in\mathbb{R}} A_{K,\xi}(t_0) \right)^\frac{1}{n-1}
	= \frac{8}{r} A_{K,\xi}^\frac{1}{n-1}\big( t_K(\xi) \big) .
\end{align*}
Finally, since $K$ is contained in a ball of radius $R$, we have
\begin{align*}
A_{K,\xi}\big( t_K(\xi) \big) \leq \frac{\pi^\frac{n-1}{2}}{\Gamma\left( \frac{n+1}{2} \right) } R^{n-1} .
\end{align*}
Combining these inequalities gives
\begin{align*}
\left| A_{K,\xi}(t) - A_{K,\xi}(s) \right| \leq L(n) \, R^{n-1} \, r^{-1} \, |t-s|
\end{align*}
for all $s,t\in [-r/2,r/2]$ and $\xi\in S^{n-1}$.
\end{proof}


We now prove two lemmas that will be the core of the proof of Theorem~\ref{MainResult_1}.

\begin{lemma}\label{integral of parallel}
Let $K$ be a convex body in $\mathbb{R}^n$ contained in a ball of radius $R$, and containing a ball of radius $r$, where both balls are centred at the origin. Let $\lbrace K_\delta \rbrace_{0<\delta<1}$ be as in Lemma \ref{approx}. If there exists $0<\varepsilon < \frac{r^2}{16}$ so that
\begin{align*}
\rho(CK, IK) \leq \varepsilon,
\end{align*}
then, for $\delta>0$ small enough,
\begin{align*}
&\int_{S^1} \big| A_{K_\delta,\xi}'(0)\big|\, d\xi \leq \left( 6\pi + \frac{32\, \pi}{ \sqrt{3} \, r } \right) \sqrt{\varepsilon}
	\ &\mbox{ when } n=2, \\
&\int_{S^{n-1}} \big| A_{K_\delta,\xi}'(0)\big|^2\, d\xi
	\leq C(n) \left( \sqrt{\varepsilon} + \frac{R^{2n-4}}{ r } + \frac{R^{3n-3}}{r^{n+2}} \right) \sqrt{\varepsilon}
	\ &\mbox{ when } n\geq 3.
\end{align*}
Here, $C(n)>0$ are constants depending only on the dimension.
\end{lemma}

\begin{proof}
By Lemma \ref{approx}, we may choose $0< \alpha <1/2$ small enough so that for every $0<\delta<\alpha$,
\begin{align*}
\sup_{\xi\in S^{n-1}} \sup_{|t|\leq r/4} \big| A_{K,\xi}(t) - A_{K_\delta,\xi}(t) \big| \leq \varepsilon .
\end{align*}
We first show that for each $0<\delta<\alpha$ and $\xi\in S^{n-1}$, there exists a number $c_\delta(\xi)$ with $| c_\delta(\xi)| \leq \sqrt{\varepsilon}$ for which
\begin{align*}
\big| A_{K_\delta,\xi}'\big( c_\delta(\xi) \big) \big| \leq 3 \sqrt{\varepsilon} .
\end{align*}
Indeed, if $\xi\in S^{n-1}$ is such that $|t_{K_\delta}(\xi)| \leq \sqrt{\varepsilon}$, then
\begin{align*}
A_{K_\delta,\xi}'\big( t_{K_\delta}(\xi) \big) = 0 ,
\end{align*}
and we may take $c_\delta(\xi) = t_{K_\delta}(\xi)$.

Assume $\xi\in S^{n-1}$ is such that $|t_{K_\delta}(\xi)|>\sqrt{\varepsilon}$. Letting $s$ denote the sign of $t_{K_\delta}(\xi)$, we have
\begin{align*}
&\big| A_{K_\delta,\xi}(s\sqrt{\varepsilon}) - A_{K_\delta,\xi}(0) \big|
	=  A_{K_\delta,\xi}(s\sqrt{\varepsilon}) - A_{K_\delta,\xi}(0) \\
&= \Big( A_{K,\xi}(s\sqrt{\varepsilon}) - A_{K,\xi}(0) \Big)
	+ \Big( A_{K_\delta,\xi}(s\sqrt{\varepsilon}) - A_{K,\xi}(s\sqrt{\varepsilon}) \Big) \\
&\hspace{4.5cm} + \Big( A_{K,\xi}(0) - A_{K_\delta,\xi}(0) \Big) \\
&\leq \sup_{\xi\in S^{n-1}} \left| \max_{t\in\mathbb{R}} A_{K,\xi}(t) - A_{K,\xi}(0) \right|
	+ 2 \sup_{\xi\in S^{n-1}} \sup_{|t|\leq r/4} \big| A_{K,\xi}(t) - A_{K_\delta,\xi}(t) \big| \\
&\leq 3 \varepsilon .
\end{align*}
It then follows from the Mean Value Theorem that there is a number $c_\delta(\xi)$ with  $|c_\delta(\xi)|\leq \sqrt{\varepsilon}$ for  which
\begin{align*}
\big| A_{K_\delta,\xi}'\big( c_\delta(\xi)\big) \big|
	= \left| \frac{ A_{K_\delta,\xi}(s\sqrt{\varepsilon}) - A_{K_\delta,\xi}(0) }{ \sqrt{\varepsilon} - 0 } \right|
	\leq 3 \sqrt{\varepsilon} .
\end{align*}

With the numbers $c_\delta(\xi)$ as above, for the case $n=2$ we have
\begin{align}\label{dim2}
&\int_{S^1} \big| A_{K_\delta,\xi}'(0)\big| \, d\xi \nonumber \\
&\leq \int_{S^1} \left( \big| A_{K_\delta,\xi}'\big(c_\delta(\xi)\big) \big|
	+ \left| \int_{ c_\delta(\xi) }^0 A_{K_\delta,\xi}''(t) \, dt \right| \right) \, d\xi \nonumber \\
&\leq 6\pi \sqrt{\varepsilon}
	+ \int_{S^1} \int_{-\sqrt{\varepsilon}}^{\sqrt{\varepsilon}} \big| A_{K_\delta,\xi}''(t) \big| \, dt \, d\xi.
\end{align}

When $0<\delta<1/2$, $K_\delta$ is contained in a ball of radius $2R$, and contains a ball of radius $r/2$. Lemma \ref{Lipschitz} then implies
\begin{align*}
\sup_{\xi\in S^{n-1}} \sup_{t\in (-\sqrt{\varepsilon}, \sqrt{\varepsilon})}
	\big| A_{K_\delta,\xi}'(t) \big| \leq \frac{ 2 L(n) \, (2R)^{n-1}}{ r} .
\end{align*}
So, when $n\geq 3$,
\begin{align}\label{dim>2}
&\int_{S^{n-1}} \big| A_{K_\delta,\xi}'(0)\big|^2 \, d\xi \nonumber \\
&\leq \int_{S^{n-1}} \left( \big| A_{K_\delta,\xi}'\big(c_\delta(\xi)\big) \big|^2
	+ \left| \int_{ c_\delta(\xi) }^0 2 A_{K_\delta,\xi}''(t) A_{K_\delta,\xi}'(t) \, dt \right| \right) \, d\xi \nonumber \\
&\leq 9\, \omega_n \, \varepsilon 	+ \frac{ 4 L(n) \, (2R)^{n-1}}{ r} \int_{S^{n-1}}
	\int_{-\sqrt{\varepsilon}}^{\sqrt{\varepsilon}} \big| A_{K_\delta,\xi}''(t) \big| \, dt \, d\xi
\end{align}

Considering inequalities (\ref{dim2}) and (\ref{dim>2}), we still need to bound
\begin{align*}
\int_{S^{n-1}} \int_{-\sqrt{\varepsilon}}^{\sqrt{\varepsilon}} \big| A_{K_\delta,\xi}''(t) \big| \, dt \, d\xi
\end{align*}
for arbitrary $n$. Rearranging the equation
\begin{align*}
\frac{d^2}{dt^2} A_{K_\delta,\xi}^\frac{1}{n-1} (t)
	&= \frac{d}{dt}\left( \frac{1}{n-1} A_{K_\delta,\xi}^\frac{2-n}{n-1}(t) \, A_{K_\delta,\xi}'(t) \right) \\
&= \frac{2-n}{(n-1)^2} A_{K_\delta,\xi}^\frac{3-2n}{n-1}(t)\, \left( A_{K_\delta,\xi}'(t) \right)^2
	+ \frac{1}{n-1} A_{K_\delta,\xi}^\frac{2-n}{n-1}(t) \, A_{K_\delta,\xi}''(t)
\end{align*}
gives
\begin{align*}
A_{K_\delta,\xi}''(t)
	= (n-1) A_{K_\delta,\xi}^\frac{n-2}{n-1}(t) \frac{d^2}{dt^2} A_{K_\delta,\xi}^\frac{1}{n-1} (t)
	+ \frac{n-2}{n-1} \frac{ \big( A_{K_\delta,\xi}'(t)\big)^2}{ A_{K_\delta,\xi}(t)} .
\end{align*}
Brunn's Theorem implies that the second derivative of $A_{K_\delta,\xi}^\frac{1}{n-1}$ is non-positive for $|t| < r$, so
\begin{align*}
\big| A_{K_\delta,\xi}''(t) \big|
	&\leq (1-n) A_{K_\delta,\xi}^\frac{n-2}{n-1}(t) \frac{d^2}{dt^2} A_{K_\delta,\xi}^\frac{1}{n-1} (t)
	+ \frac{n-2}{n-1} \frac{ \big( A_{K_\delta,\xi}'(t)\big)^2}{ A_{K_\delta,\xi}(t)} \\
&= - A_{K_\delta,\xi}''(t) + 2 \left( \frac{n-2}{n-1} \right) \frac{ \big( A_{K_\delta,\xi}'(t)\big)^2}{ A_{K_\delta,\xi}(t)} .
\end{align*}
Because $K_\delta$ contains a ball of radius $r/2$ centred at the origin, we have
\begin{align*}
A_{K_\delta,\xi}(t) \geq \frac{1}{\Gamma\left(\frac{n+1}{2}\right)} \left(\frac{3\pi r^2}{16}\right)^\frac{n-1}{2}
\end{align*}
for $|t| \leq r/4$, and so
\begin{align*}
\frac{n-2}{n-1} \frac{ \big( A_{K_\delta,\xi}'(t)\big)^2}{ A_{K_\delta,\xi}(t)} &\leq \frac{n-2}{n-1} \, \Gamma\left( \frac{n+1}{2}\right) \,
	\left( \frac{ 2 L(n) \, (2R)^{n-1}}{ r} \right)^2 \, \left( \frac{16}{3\pi r^2}\right)^\frac{n-1}{2} \\
&= \frac{ \widetilde{L}(n) \, R^{2n-2} }{ r^{n+1} }
\end{align*}
for all $|t| \leq \sqrt{\varepsilon}$, where $\widetilde{L}(n)$ is a constant depending only on $n$. Therefore,
\begin{align}\label{dim 2 & dim>2}
&\int_{S^{n-1}} \int_{-\sqrt{\varepsilon}}^{\sqrt{\varepsilon}} \Big| A_{K_\delta,\xi}''(t)\Big| \, dt\, d\xi \nonumber \\
&\leq  \int_{S^{n-1}} \int_{-\sqrt{\varepsilon}}^{\sqrt{\varepsilon}} \Big(- A_{K_\delta,\xi}''(t)\Big) \, dt \, d\xi
	+ \frac{ 4 \, \omega_n \, \widetilde{L}(n) \, R^{2n-2} }{ r^{n+1} } \sqrt{\varepsilon} .
\end{align}

We will bound the first term on the final line above using formula (\ref{BK formula}). Letting
\begin{align*}
\widetilde{C}(n) = \frac{\Gamma\left(\frac{n}{2}\right)}{\sqrt{\pi} \, \Gamma\left(\frac{n-1}{2}\right)} ,
\end{align*}
formula (\ref{BK formula}) becomes
\begin{align*}
f_{K_\delta}(t) &= \widetilde{C}(n) \int_{S^{n-1}}
	\int_{|t|}^{\rho_{K_\delta}(\xi)} \frac{1}{r} \left( 1 - \frac{t^2}{r^2} \right)^\frac{n-3}{2} r^{n-1}\, dr\,d\xi \\
&= \widetilde{C}(n) \int_{S^{n-1}} \int_{|t|}^{\rho_{K_\delta}(\xi)} r \left( r^2 - t^2 \right)^\frac{n-3}{2} \, dr \, d\xi \\
&= \frac{\widetilde{C}(n)}{(n-1)} \int_{S^{n-1}} \left( \rho_{K_\delta}^2(\xi) - t^2 \right)^\frac{n-1}{2} \, d\xi .
\end{align*}
The derivatives of $A_{K_\delta,\xi}$ and $\left( \rho_{K_\delta}^2(\xi) - t^2 \right)^\frac{n-1}{2}$ are bounded on $(-\sqrt{\varepsilon},\sqrt{\varepsilon})$ uniformly with respect to $\xi\in S^{n-1}$, so
\begin{align*}
f_{K_\delta}'(t) = \frac{1}{\omega_n} \int_{S^{n-1}} A_{K_\delta,\xi}'(t) \, d\xi
	= -\widetilde{C}(n) \, t \int_{S^{n-1}} \left( \rho_{K_\delta}^2(\xi) - t^2 \right)^\frac{n-3}{2}\, d\xi .
\end{align*}
Observing $\widetilde{C}(2) = \pi^{-1}$, and using that $0<\varepsilon <r^2/16$ and $r/2\leq \rho_{K_\delta}\leq 2R$ for $\delta<1/2$, we have
\begin{align*}
&\left| \int_{S^{n-1}} A_{K_\delta,\xi}'(\pm \sqrt{\varepsilon}) \, d\xi \right| \\
&= \omega_n \left| f_{K_\delta}'(\pm \sqrt{\varepsilon}) \right|
	= \widetilde{C}(n) \, \omega_n \sqrt{\varepsilon}
	\int_{S^{n-1}} \left( \rho_{K_\delta}^2(\xi) - \varepsilon \right)^\frac{n-3}{2}\, d\xi \\
&\leq
	\begin{cases}
		16 \, \pi \left(\sqrt{3} \, r\right)^{-1} \sqrt{\varepsilon} \ &\mbox{ if } n=2, \\
		\widetilde{C}(n) \, \omega_n^2 \, (2R)^{n-3} \sqrt{\varepsilon} \ &\mbox{ if } n\geq 3 .
	\end{cases}
\end{align*}
This implies
\begin{align}\label{dim combined}
\left| \int_{S^{n-1}} \int_{-\sqrt{\varepsilon}}^{\sqrt{\varepsilon}} - A_{K_\delta,\xi}''(t) \, dt \, d\xi \right|
	&= \left| \int_{S^{n-1}} \big( A_{K_\delta,\xi}'(-\sqrt{\varepsilon})
	- A_{K_\delta,\xi}'(\sqrt{\varepsilon}) \big) \, d\xi \right| \nonumber \\
&\leq
	\begin{cases}
		32 \, \pi \left(\sqrt{3} \, r\right)^{-1} \sqrt{\varepsilon} \ &\mbox{ if } n=2, \\
		2 \, \widetilde{C}(n) \, \omega_n^2 \, (2R)^{n-3} \sqrt{\varepsilon} \ &\mbox{ if } n\geq 3 .
	\end{cases}
\end{align}

Noting that $\widetilde{L}(2) = 0$, inequalities (\ref{dim2}), (\ref{dim 2 & dim>2}), and (\ref{dim combined}) give
\begin{align*}
\int_{S^1} \big| A_{K_\delta,\xi}'(0)\big| \, d\xi \leq \left( 6\pi + \frac{32\pi}{ \sqrt{3} \, r } \right) \sqrt{\varepsilon}
\end{align*}
when $n=2$. For $n\geq 3$, inequalities (\ref{dim>2}), (\ref{dim 2 & dim>2}), and (\ref{dim combined}) give
\begin{align*}
&\int_{S^{n-1}} \big| A_{K_\delta,\xi}'(0)\big|^2\, d\xi \leq C(n) \left( \sqrt{\varepsilon}
	+ \frac{R^{2n-4}}{ r } + \frac{R^{3n-3}}{r^{n+2}} \right) \sqrt{\varepsilon} ,
\end{align*}
where $C(n)$ is a constant depending on $n$.
\end{proof}


\begin{lemma}\label{keylemma}
Let $K$ and $L$ be infinitely smooth convex bodies in $\mathbb{R}^n$ which are contained in a ball of radius $R$, and contain a ball of radius $r$, where both balls are centred at the origin. Let $p\in (0,n)$. If $\varepsilon>0$ is such that
\begin{align*}
\left\Vert I_p\left( \Vert\xi\Vert_K^{-n+p} - \Vert\xi\Vert_L^{-n+p} \right) \right\Vert_2 \leq \varepsilon ,
\end{align*}
then when $n\leq 2p$,
\begin{align*}
\rho(K,L) \leq C(n,p) \, R^2 r^\frac{-3n-1+2p}{n+1} \varepsilon^\frac{2}{n+1},
\end{align*}
and when $n>2p$,
\begin{align*}
\rho(K,L) \leq C(n,p) \, R^2 r^\frac{-3n-1+2p}{n+1}
	\left( \varepsilon^2 + \frac{R^{2(n+1-p)}}{ r^2 } \right)^\frac{n-2p}{(n+2-2p)(n+1)} \varepsilon^\frac{4}{(n+2-2p)(n+1)} .
\end{align*}
Here, $\|\cdot\|_2$ denotes the norm on $L^2(S^{n-1})$, and $C(n,p)>0$ are constants depending on the dimension and $p$.
\end{lemma}

\begin{proof}
Define the function
\begin{align*}
f(\xi) := \Vert\xi\Vert_K^{-n+p} - \Vert\xi\Vert_L^{-n+p}
\end{align*}
on $S^{n-1}$. Towards bounding the radial distance between $K$ and $L$ by $\Vert f\Vert_2$, the $L^2(S^{n-1})$ norm of $f$, note that the identity
\begin{align*}
\rho_K(\xi) - \rho_L(\xi) = \rho_K(\xi) \rho_L(\xi)\big( \Vert\xi\Vert_L - \Vert\xi\Vert_K\big)
\end{align*}
implies
\begin{align*}
\big| \rho_K(\xi) - \rho_L(\xi)\big| \leq R^2 \big| \Vert\xi\Vert_K - \Vert\xi\Vert_L\big| .
\end{align*}
By Theorem \ref{Vitale}, we have
\begin{align*}
\delta_\infty (K^\circ,L^\circ) \leq C(n) D^\frac{n-1}{n+1} \big( \delta_2(K^\circ,L^\circ)\big)^\frac{2}{n+1},
\end{align*}
where $C(n)>0$ is a constant depending on $n$, and $D$ is the diameter of $K^\circ\cup L^\circ$. Both $K^\circ$ and $L^\circ$ are contained in a ball of radius $r^{-1}$ centred at the origin. We then have $D\leq 2r^{-1}$, and
\begin{align*}
\sup_{\xi\in S^{n-1}} \big| \Vert\xi\Vert_K-\Vert\xi\Vert_L\big| \leq C(n) r^\frac{1-n}{n+1}
	\left( \int_{S^{n-1}} \big( \Vert\xi\Vert_K - \Vert\xi\Vert_L\big)^2\, d\xi \right)^\frac{1}{n+1}
\end{align*}
for some new constant $C(n)$. There exists a function $g:S^{n-1}\rightarrow\mathbb{R}$ such that
\begin{align*}
\big( \Vert\xi\Vert_K - \Vert\xi\Vert_L\big) g(\xi) = \Vert\xi\Vert_K^{-n+p} - \Vert\xi\Vert_L^{-n+p} .
\end{align*}
If $\xi\in S^{n-1}$ is such that $\Vert\xi\Vert_K\neq\Vert\xi\Vert_L$, then an application of the Mean Value Theorem to the function $t^{-n+p}$ on the interval bounded by $\Vert\xi\Vert_K$ and $\Vert\xi\Vert_L$ gives
\begin{align*}
| g(\xi) | \geq (n-p) \left( \max\big\lbrace \Vert\xi\Vert_K,\Vert\xi\Vert_L \big\rbrace \right)^{-n-1+p}
	\geq (n-p) r^{n+1-p} .
\end{align*}
Therefore,
\begin{align*}
\big| \Vert\xi\Vert_K - \Vert\xi\Vert_L\big| \leq (n-p)^{-1} r^{-n-1+p} |f(\xi)| .
\end{align*}
Combining the above inequalities, we get
\begin{align}\label{radial to L2}
\sup_{\xi\in S^{n-1}} \big|\rho_K(\xi)-\rho_L(\xi)\big| \leq C(n,p) R^2 r^\frac{-3n-1+2p}{n+1} \Vert f\Vert_2^\frac{2}{n+1} ,
\end{align}
for some constant $C(n,p)$.

We now compare the $L^2$ norm of $f$ to that of $I_p(f)$ by considering two separate cases based on the dimension $n$, as in the proof of Theorem 3.6 in \cite{Good&Yask2009}. In both cases, we let $\sum_{m=0}^\infty Q_m$ be the condensed harmonic expansion for $f$, and let $\lambda_m(n,p)$ be the eigenvalues from Lemma \ref{eigen}. As in \cite{Good&Yask2009}, the condensed harmonic expansion for $I_pf$ is then given by $\sum_{m=0}^\infty \lambda_m(n,p) Q_m$.

Assume $n\leq 2p$. An application of Stirling's formula to the equations given in Lemma \ref{eigen} shows that $\lambda_m(n,p)$ diverges to infinity as $m$ approaches infinity. The eigenvalues are also non-zero, so there is a constant $C(n,p)$ such that $C(n,p) |\lambda_m(n,p)|^2$ is greater than one for all $m$. Therefore,
\begin{align*}
\Vert f\Vert_2^2 &= \sum_{m=0}^\infty \Vert Q_m\Vert_2^2 \\
&\leq C(n,p) \sum_{m=0}^\infty \big|\lambda_m(n,p)\big|^2\Vert Q_m\Vert_2^2
	= C(n,p) \Vert I_p(f)\Vert_2^2 \leq C(n,p) \varepsilon^2 .
\end{align*}
Combining this inequality with (\ref{radial to L2}) gives the first estimate in the theorem.

Assume $n>2p$. H{\"o}lder's inequality gives
\begin{align*}
&\Vert f\Vert_2^2 = \sum_{m=0}^\infty \Vert Q_m\Vert_2^2 \\
&=\sum_{m=0}^\infty \left( \big| \lambda_m(n,p)\big|^\frac{4}{n+2-2p} \left\Vert Q_m\right\Vert_2^\frac{4}{n+2-2p} \right)
	\cdot \left( \big| \lambda_m(n,p)\big|^\frac{-4}{n+2-2p} \left\Vert Q_m\right\Vert_2^\frac{2n-4p}{n+2-2p} \right) \\
&\leq \left( \sum_{m=0}^\infty \big| \lambda_m(n,p)\big|^2 \left\Vert Q_m\right\Vert_2^2 \right)^\frac{2}{n+2-2p}
	\left( \sum_{m=0}^\infty \big| \lambda_m(n,p)\big|^\frac{-4}{n-2p} \left\Vert Q_m\right\Vert_2^2 \right)^\frac{n-2p}{n+2-2p},
\end{align*}
where we again note that the eigenvalues are all non-zero. It follows from Lemma \ref{eigen} and Stirling's formula that there is a constant $C(n,p)$ such that
\begin{align*}
\big| \lambda_m(n,p)\big|^\frac{-4}{n-2p} \leq C(n,p) m^2
\end{align*}
for all $m\geq 1$, and
\begin{align*}
\big| \lambda_0(n,p)\big|^\frac{-4}{n-2p} \leq C(n,p) .
\end{align*}
Using the identity
\begin{align}\label{2smooth}
\Vert\nabla_o f\Vert_2^2 = \sum_{m=1}^\infty m(m+n-2) \Vert Q_m\Vert_2^2
\end{align}
given by Corollary 3.2.12 in \cite{Groemer}, we then have
\begin{align*}
\Vert f\Vert_2^2 &\leq C(n,p) \left( \left\Vert I_p(f) \right\Vert_2^2 \right)^\frac{2}{n+2-2p}
	\left( \Vert Q_0\Vert_2^2 + \Vert \nabla_o f\Vert_2^2 \right)^\frac{n-2p}{n+2-2p} .
\end{align*}
The Minkowski functional of a convex body is the support function of the corresponding polar body, so
\begin{align*}
\nabla_o \Vert\xi\Vert_K^{-n+p} = (-n+p) \Vert\xi\Vert_K^{-n-1+p}\, \nabla_o h_{K^\circ}(\xi) .
\end{align*}
Because $K^\circ$ is contained in a ball of radius $r^{-1}$, it follows from Lemma 2.2.1 in \cite{Groemer} that
\begin{align*}
|\nabla_o h_{K^\circ}(\xi)| \leq 2r^{-1}
\end{align*}
for all $\xi\in S^{n-1}$. We now have
\begin{align*}
\left\Vert\nabla_o\Vert\xi\Vert_K^{-n+p}\right\Vert_2^2
	\leq 4 (n-p)^2 R^{2(n+1-p)} r^{-2} \omega_n .
\end{align*}
This constant bounds the squared $L^2$ norm of $\nabla_o \Vert\xi\Vert_L^{-n+p}$ as well, so
\begin{align*}
\big\Vert\nabla_o f\big\Vert_2^2 &\leq 16 (n-p)^2 R^{2(n+1-p)} r^{-2} \omega_n .
\end{align*}
Therefore,
\begin{align*}
\Vert f\Vert_2^2 \leq C(n,p) \varepsilon^\frac{4}{n+2-2p} \left( \varepsilon^2 + R^{2(n+1-p)} r^{-2} \right)^\frac{n-2p}{n+2-2p} ,
\end{align*}
where the constant $C(n,p)>0$ is different from before. This inequality with (\ref{radial to L2}) gives the second estimate in the theorem.
\end{proof}

\section{Proofs of Stability Results}

We are now ready to prove our stability results. 

\begin{proof}[Proof of Theorem \ref{MainResult_1}]
Let $\lbrace K_\delta \rbrace_{0<\delta<1}$ be the family of smooth convex bodies from Lemma \ref{approx}. We will show that $\rho(K_\delta, -K_\delta)$ is small for $0<\delta<\alpha$, where $\alpha$ is the constant from the proof of Lemma \ref{integral of parallel}. The bounds in the theorem will then follow from
\begin{align*}
\rho(K,-K) \leq \lim_{\delta\rightarrow 0} \big( 2\rho(K,K_\delta) + \rho(K_\delta, -K_\delta) \big)
	= \lim_{\delta\rightarrow 0} \rho(K_\delta,-K_\delta) .
\end{align*}

We begin by separately considering the case $n=2$. Let the radial function $\rho_{K_\delta}$ be a function of the angle measured counter-clockwise from the positive horizontal axis. For any $\xi\in S^1$, let the angles $\phi_1$ and $\phi_2$ be functions of $t\in (-r,r)$ as indicated in Figure \ref{K}. If $\xi$ corresponds to the angle $\theta$, then the parallel section function for $K_\delta$ may be written as
\begin{align*}
A_{K_\delta,\theta}(t) = \rho_{K_\delta}(\theta + \phi_1) \, \sin \phi_1 + \rho_{K_\delta}(\theta -\phi_2) \, \sin \phi_2 .
\end{align*}
Implicit differentiation of
\begin{align*}
\cos \phi_j = \frac{t}{\rho_{K_\delta}\left(\theta - (-1)^j\phi_j\right)} \ \ (j=1,2)
\end{align*}
gives
\begin{align*}
\frac{d\phi_j}{dt}\Big|_{t=0} = \frac{(-1)}{ \rho_{K_\delta}\left( \theta - (-1)^j \frac{\pi}{2}\right) } ,
\end{align*}
so
\begin{align*}
A_{K_\delta,\theta}'(0) =
	- \frac{ \rho_{K_\delta}'\left( \theta + \frac{\pi}{2}\right) }{ \rho_{K_\delta}\left( \theta + \frac{\pi}{2}\right) }
	+ \frac{ \rho_{K_\delta}'\left( \theta - \frac{\pi}{2}\right) }{ \rho_{K_\delta}\left( \theta - \frac{\pi}{2}\right) } .
\end{align*}

\begin{figure}\label{K}
\center \includegraphics[scale=0.75]{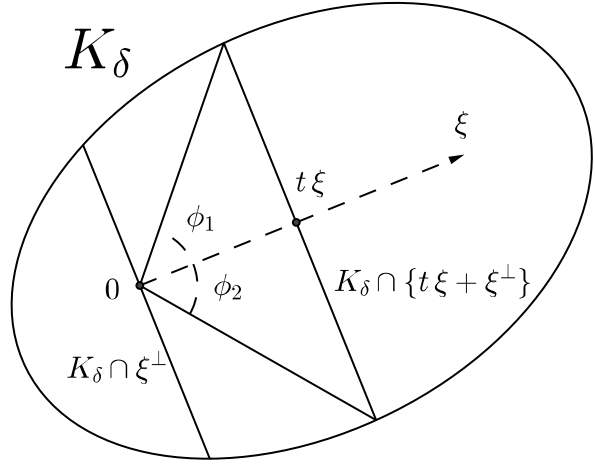}
\caption{$K_\delta$ is a convex body in $\mathbb{R}^2$, and $\xi\in S^1$.}
\end{figure}

\noindent Since $f(\phi) := \rho_{K_\delta}(\phi + \pi/2) - \rho_{K_\delta}(\phi - \pi/2)$ is a continuous function on $[0,\pi]$ with
\begin{align*}
f(0) = \rho_{K_\delta}(\pi/2) - \rho_{K_\delta}(-\pi/2) = - \big( \rho_{K_\delta}(-\pi/2) - \rho_{K_\delta}(\pi/2) \big)
	= - f(\pi) ,
\end{align*}
there exists an angle $\theta_0\in [0,\pi]$ such that $\rho_{K_\delta}( \theta_0 + \pi/2 ) = \rho_{K_\delta}( \theta_0 - \pi/2)$. With this $\theta_0$, we get the inequality
\begin{align*}
\left| \int_{\theta_0}^\theta
	\left(- \frac{ \rho_{K_\delta}'\left( \phi + \frac{\pi}{2}\right) }{ \rho_{K_\delta}\left( \phi + \frac{\pi}{2}\right) }
	+ \frac{ \rho_{K_\delta}'\left( \phi - \frac{\pi}{2}\right) }{ \rho_{K_\delta}\left( \phi - \frac{\pi}{2}\right) }\right)
	\, d\phi \right|
	\leq \int_0^{2\pi} \left| A_{K_\delta,\phi}'(0) \right| \, d\phi .
\end{align*}
Integrating the left side of this inequality, and applying Lemma \ref{integral of parallel} to the right side, gives
\begin{align*}
\left| \log\left(
	\frac{ \rho_{K_\delta}\left( \theta - \frac{\pi}{2}\right) }{ \rho_{K_\delta}\left( \theta + \frac{\pi}{2}\right) }
	\right) \right|
	\leq \left( 6\pi + \frac{32\pi}{ \sqrt{3} \, r } \right) \sqrt{\varepsilon} .
\end{align*}
This implies
\begin{align*}
1 - \exp\left[ \left( 6\pi + \frac{32\pi}{ \sqrt{3} \, r } \right) \sqrt{\varepsilon} \right]
	&\leq \exp\left[ - \left( 6\pi + \frac{32\pi}{ \sqrt{3} \, r } \right) \sqrt{\varepsilon} \right] - 1 \\
&\leq \frac{ \rho_{K_\delta}\left( \theta - \frac{\pi}{2}\right) }{ \rho_{K_\delta}\left( \theta + \frac{\pi}{2}\right) } - 1 \\
&\leq \exp\left[ \left( 6\pi + \frac{32\pi}{ \sqrt{3} \, r } \right)\ \sqrt{\varepsilon} \right] - 1 .
\end{align*}
It follows that
\begin{align*}
- 2 \left( \exp\left[ \left( 6\pi + \frac{32\pi}{ \sqrt{3} \, r } \right) \sqrt{\varepsilon} \right] - 1 \right) R
	&\leq \rho_{K_\delta}\left( \theta - \frac{\pi}{2}\right) - \rho_{K_\delta}\left( \theta + \frac{\pi}{2}\right) \\
&\leq 2 \left( \exp\left[ \left( 6\pi + \frac{32\pi}{ \sqrt{3} \, r } \right) \sqrt{\varepsilon} \right] - 1 \right) R ,
\end{align*}
since $K_\delta$ is contained in a ball of radius $2R$. Viewing $\rho_{K_\delta}$ again as a function of vectors, we have
\begin{align*}
\sup_{\xi\in S^1} \left\vert \rho_{K_\delta}(\xi) - \rho_{K_\delta}(-\xi)\right\vert
	\leq 2 \left( \exp\left[ \left( 6\pi + \frac{32\pi}{ \sqrt{3} \, r } \right) \sqrt{\varepsilon} \right] - 1 \right) R .
\end{align*}
The inequality $e^t - 1 \leq 2t$ is valid when $0<t<1$; therefore, if
\begin{align*}
\varepsilon < \left( \frac{\sqrt{3} \, r}{ 6\sqrt{3}\pi r + 32\pi } \right)^2 ,
\end{align*}
then
\begin{align*}
\sup_{\xi\in S^1} \left\vert \rho_{K_\delta}(\xi) - \rho_{K_\delta}(-\xi)\right\vert
	\leq \left( 24\pi + \frac{128\pi}{ \sqrt{3} \, r } \right) R  \sqrt{\varepsilon} .
\end{align*}

Consider the case when $n>2$. For $K_\delta$ with $p=1$, Equation (\ref{PS Func}) becomes
\begin{align*}
I_2\left( \Vert x\Vert_{K_\delta}^{-n+2} - \Vert -x\Vert_{K_\delta}^{-n+2} \right) (\xi)
	=  -2\pi i\, (n-2) A_{K_\delta,\xi}'(0) ,
\end{align*}
so
\begin{align*}
&\left\Vert I_2\left( \Vert x\Vert_{K_\delta}^{-n+2} - \Vert x\Vert_{-K_\delta}^{-n+2} \right)\right\Vert_2
	= 2\pi (n-2) \left( \int_{S^{n-1}} \big| A_{K_\delta,\xi}'(0) \big|^2 \, d\xi \right)^{\frac{1}{2}} \\
&\leq \widetilde{C}(n) \left( \sqrt{\varepsilon} + \frac{R^{2n-4}}{ r } + \frac{R^{3n-3}}{r^{n+2}} \right)^\frac{1}{2}
	\varepsilon^\frac{1}{4}
\end{align*}
by Lemma \ref{integral of parallel}. Finally, by Lemma \ref{keylemma},
\begin{align*}
\rho(K_\delta,-K_\delta) \leq C(n) \frac{ R^2 }{ r^\frac{3n-3}{n+1} }
	\left( \sqrt{\varepsilon} + \frac{R^{2n-4}}{ r } + \frac{R^{3n-3}}{r^{n+2}} \right)^\frac{1}{n+1}
	\varepsilon^\frac{1}{2(n+1)}
\end{align*}
when $n = 3 \mbox{ or } 4$, and
\begin{align*}
\rho(K_\delta,-K_\delta)
	\leq C(n) &\left[ \left( \sqrt{\varepsilon} + \frac{R^{2n-4}}{ r } + \frac{R^{3n-3}}{r^{n+2}} \right) \sqrt{\varepsilon}
	+ \frac{ R^{2(n-1)} }{ r^2 } \right]^\frac{n-4}{ (n-2)(n+1) } \\
&\cdot \left( \sqrt{\varepsilon} + \frac{R^{2n-4}}{ r } + \frac{R^{3n-3}}{r^{n+2}} \right)^\frac{2}{(n-2)(n+1)}
	\frac{ R^2 \varepsilon^\frac{1}{(n-2)(n+1)} }{ r^\frac{3n-3}{n+1} }
\end{align*}
when $n\geq 5$, where $C(n)>0$ are constants depending on the dimension.
\end{proof}

We now present the proof of our second stability result.

\begin{proof}[Proof of Theorem \ref{MainResult_2}]
Apply Lemma \ref{approx} to $K$ and $L$; let $\lbrace K_\delta\rbrace_{0<\delta<1}$ and $\lbrace L_\delta\rbrace_{0<\delta<1}$ be the resulting families of smooth convex bodies. For each $\delta$, define the constant
\begin{align*}
\varepsilon_\delta := \sup_{\xi\in S^{n-1}} \left| A_{K_\delta,\xi}^{(p)}(0) - A_{K,\xi}^{(p)}(0)\right|
	+ \sup_{\xi\in S^{n-1}} \left| A_{L_\delta,\xi}^{(p)}(0) - A_{L,\xi}^{(p)}(0)\right| + \varepsilon .
\end{align*}
Defining the auxiliary function
\begin{align*}
f_\delta(\xi) := \Vert\xi\Vert_{K_\delta}^{-n+1+p} - \Vert\xi\Vert_{L_\delta}^{-n+1+p},
\end{align*}
we have
\begin{align*}
&\cos\left( \frac{p\pi}{2}\right) I_{1+p}\big( f_\delta(x) + f_\delta(-x) \big) (\xi)
	+ i \sin\left( \frac{p\pi}{2}\right) I_{1+p}\big( f_\delta(x) - f_\delta(-x) \big) (\xi) \\
&= 2\pi (n-1-p) \Big( A_{K_\delta,\,\xi}^{(p)}(0) - A_{L_\delta,\,\xi}^{(p)}(0) \Big)
\end{align*}
from Equation (\ref{PS Func}). The function of $\xi$ on the left side of this equality is split into its even and odd parts, because $I_{1+p}$ preserves even and odd symmetry. Therefore,
\begin{align*}
&\frac{\cos\left( \frac{p\pi}{2}\right)}{\pi(n-1-p)} I_{1+p}\big( f_\delta(x) + f_\delta(-x) \big) (\xi) \\
&\qquad = \Big( A_{K_\delta,\,\xi}^{(p)}(0) - A_{L_\delta,\,\xi}^{(p)}(0)  \Big)
	+ \Big( A_{K_\delta,\,-\xi}^{(p)}(0) - A_{L_\delta,\,-\xi}^{(p)}(0) \Big)
\end{align*}
and
\begin{align*}
&\frac{i \sin\left( \frac{p\pi}{2}\right)}{\pi(n-1-p)} I_{1+p}\big( f_\delta(x) - f_\delta(-x) \big) (\xi)   \\
&\qquad = \Big( A_{K_\delta,\,\xi}^{(p)}(0) - A_{L_\delta,\,\xi}^{(p)}(0) \Big)
	- \Big( A_{K_\delta,\,-\xi}^{(p)}(0) - A_{L_\delta,\,-\xi}^{(p)}(0) \Big)
\end{align*}
By the definition of $\varepsilon_\delta$,
\begin{align*}
\Big| I_{1+p} \big( 2f_\delta\big) (\xi) \Big| &\leq \bigg| I_{1+p}\Big( f_\delta(x) + f_\delta(-x) \Big) (\xi) \bigg|
	+ \bigg| I_{1+p}\Big( f_\delta(x) - f_\delta(-x) \Big) (\xi) \bigg| \\
&\leq \frac{2\pi (n-1-p)}{\cos\left( ^{p\pi}/_{2}\right)} \,\varepsilon_\delta
	+ \frac{2\pi (n-1-p)}{\sin\left( ^{p\pi}/_{2}\right)} \,\varepsilon_\delta ,
\end{align*}
which implies
\begin{align*}
\left\Vert I_{1+p}(f_\delta)\right\Vert_2
	\leq \pi \sqrt{\omega_n} \, (n-1-p)
	\Big( \big|\sec\left( ^{p\pi}/_2\right)\big| + \big|\csc\left( ^{p\pi}/_2\right)\big| \Big) \varepsilon_\delta .
\end{align*}
Both $K_\delta$ and $L_\delta$ are contained in a ball of radius $2R$ when $0<\delta<1/2$, and contain a ball of radius $r/2$. It now follows from Lemma \ref{keylemma} that
\begin{align*}
\rho(K_\delta, L_\delta) \leq C(n,p) \, R^2 r^\frac{-3n+1+2p}{n+1} \varepsilon_\delta^\frac{2}{n+1}
\end{align*}
when $n\leq 2p+2$, and
\begin{align*}
\rho(K_\delta, L_\delta) \leq C(n,p) \, R^2 r^\frac{-3n+1+2p}{n+1}
	\left( \varepsilon_\delta^2 + \frac{ R^{2(n-p)} }{ r^2 }\right)^\frac{n-2-2p}{(n-2p)(n+1)}
	\varepsilon_\delta^\frac{4}{(n-2p)(n+1)}
\end{align*}
when $n> 2p+2$, where $C(n,p)>0$ are constants depending on the dimension and $p$. Finally, the bounds in the theorem statement follow from the observations
\begin{align*}
\rho(K,L) \leq \lim_{\delta\rightarrow 0} \Big( \rho(K,K_\delta) + \rho(L,L_\delta) + \rho(K_\delta, L_\delta) \Big)
	= \lim_{\delta\rightarrow 0} \rho(K_\delta, L_\delta) ,
\end{align*}
and $\lim_{\delta\rightarrow 0} \varepsilon_\delta = \varepsilon$.
\end{proof}

\end{document}